\documentclass[11pt]{amsart}
\usepackage{amsfonts, amsmath, amssymb, color}
\usepackage{amsthm}
\usepackage{booktabs}
\usepackage{float}
\usepackage{hhline}
\usepackage{lipsum}
\usepackage{lscape}
\usepackage{subfig}
\usepackage{textcomp}
\usepackage{tikz}
\usetikzlibrary{decorations.pathmorphing,shapes,arrows,positioning}
\usepackage{xcolor}
\usepackage{young}
\usepackage{youngtab}
\usepackage{ytableau}
\allowdisplaybreaks[4]
\theoremstyle{plain}
\newtheorem{thm}{Theorem}[section]
\newtheorem{lem}[thm]{Lemma}
\newtheorem{prop}[thm]{Proposition}
\newtheorem{cor}[thm]{Corollary}
\newtheorem{rem}[thm]{Remark}

\theoremstyle{definition}
\newtheorem{defn}[thm]{Definition}
\newtheorem{exmp}[thm]{Example}

\newcommand{\la}{\lambda}

\newcommand{\tabincell}[2]{\begin{tabular}{@{}#1@{}}#2\end{tabular}}  %表格自动换行
\numberwithin{equation}{section} \errorcontextlines=0

\begin{document}
\title{The Green polynomials via vertex operators}
\author{Naihuan Jing}
\address{Department of Mathematics, North Carolina State University, Raleigh, NC 27695, USA}
\email{jing@ncsu.edu}
\author{Ning Liu}
\address{School of Mathematics, South China University of Technology,
Guangzhou, Guangdong 510640, China}
\email{mathliu123@outlook.com}
%\thanks{{\scriptsize
%\hskip -0.4 true cm MSC (2010): Primary: 81P40; Secondary: 81Qxx.
%$*$Corresponding author, jing@ncsu.edu}}

\begin{abstract}
An iterative formula for the Green polynomial is given using the vertex operator realization of the Hall-Littlewood function.
Based on this, (1) a general combinatorial formula of the Green polynomial is given; (2) several compact formulas are given for
Green's polynomials associated with upper partitions of length $\leq 3$ and the diagonal lengths $\leq 3$;
(3) a Murnaghan-Nakayama type formula for the Green polynomial is obtained; and (4)
an iterative formula is derived for the bitrace of the finite general linear group $G$ and the Iwahori-Hecke algebra of type $A$ on the permutation module of $G$ by its Borel subgroup.
\end{abstract}

\maketitle

\section{Introduction}

The Green polynomials $Q^{\la}_{\mu}(q)$ were introduced by Green \cite{G} to compute irreducible characters of the finite general linear group
$\mathrm{GL}_n(\mathbb F_q)$. When $q=\infty$, they are exactly the irreducible character value $\chi^{\la}(C_{\mu})$ of the symmetric group $S_n$.
According to Hotta and Springer \cite{HS}, the $t^i$-coefficients $\psi^{\mu, i}$ of  $Q^{\la}_{\mu}(t)$ are certain characters of $S_n$
 that afford the $S_n$-action on the rational cohomology
$H^*(X_{\mu})$ of the variety $X_{\mu}$, the subvariety fixed by the unipotent elements of type $\mu$ of the flag variety.

The Green polynomial or its variant $X^{\la}_{\mu}(t)=t^{n(\mu)}Q^{\la}_{\mu}(t^{-1})$ is defined as the transition coefficient of the power-sum symmetric function $p_{\mu}$ in terms of the Hall-Littlewood symmetric function $P_{\la}(t)$ \cite{M}. Let $f_{\la\mu}^{\nu}(t)$ be the structure constants (Hall polynomial) of the Hall algebra generated by the $P_{\la}$. By Green's original definition, the Green polynomial can be written as a sum of products of lower degree Green polynomials with weights
$f_{\la\mu}^{\nu}(t)$:
\begin{equation}
Q^{\nu}_{\rho\cup\tau}(t)=\sum_{\la, \mu}f^{\nu}_{\la\mu}(t)Q^{\la}_{\rho}(t)Q^{\mu}_{\tau}(t).
\end{equation}
Based on this iteration, Green has given a table for $n\leq 5$. Morris \cite{Mor1} used an implicit iteration of the Kostka-Foulkes polynomial
to provide a table for $n=6, 7$.
%Both methods The latter method has shown improvement, but it is less direct and usually complicated in practical operation.
As far as we know, no explicit formula is known for $Q^{\la}_{\mu}(t)$ in the general case.

Lascoux-Leclerc-Thibon \cite{LLT} has proved a formula (LLT) for the Green polynomials at roots of unity, conjectured by Morris-Sultana \cite{MS}. Morita \cite{Mo}
has generalized the LLT formula and given a formula for $Q^{\la}_{\mu}(\omega)$ for $\la$ being hook-shaped at a root of unity $\omega$.

Recently, Bryan and one of us \cite{BJ} have used vertex operators to derive a direct iterative formula for the Kostka-Foulkes polynomial.
In the first part of the paper,
we will use the same idea to give an iterative formula for the Green polynomials. Using this formula, we are able to
recover all previously known compact formulas for the Green polynomials and offer some new ones. First of all, we obtain
a general combinatorial formula: %for the Green polynomial $X^{\la}_{\mu}(t)$
\begin{equation}
X^{\la}_{\mu}(t)=\sum_{\{\rho^i\}, \{\tau^i\}}\prod^{l(\la)-1}_{j=1}\frac{(-1)^{l(\rho^{(j)})}}{z_{\rho^{(j)}}(t)}
\end{equation}
summed over interlacing sequences $\{\tau^i\}$ and $\{\rho^i\}$ of partitions (see Theorem \ref{t:Formula}).
In particular, this also gives a combinatorial formula of the irreducible character value for the symmetric group
$S_n$ ($t=0$ or $q=\infty$).

Special cases of our general formula recover previously known formulas for $X^{\la}_{\mu}(t)$, In fact, our compact formula
$Q^{(k, 1^{n-k})}_{\mu}(t)$ for all values of $t$ recovers and generalizes Morita's formula. As examples, we
check that our formula recovers the known ones for $Q^{\la}_{(1^n)}(t)$ and $Q^{(1^n)}_{\mu}(t)$. Moreover, we also derive several general formulas for the Green polynomials, namely, compact formulas of $Q^{\la}_{\mu}(t)$ for $l(\mu)\leq 3$
as well as with Frobenius diagonal of $\mu$ $\leq 3$ are obtained.

Our method relies upon applications of dual vertex operators on the vertex operator realization of the Hall-Littlewood
functions \cite{Jing1} and straightening out the general Hall-Littlewood operators associated with compositions to those with
partitions. One application is to derive a Murnaghan-Nakayama formula for the Green polynomial in the general case.

The second part of the paper deals with an important application of our method. Let $G=\mathrm{GL}_n(\mathbb F_q)$ and $H_n(q)$
the Iwahori-Hecke algebra in type $A$. Let $B$ be the upper Borel subgroup of $G$, the algebra $H_n(q)$ is naturally realized via the
permutation module $\mathrm{Ind}^G_B 1$, where $1$ is the trivial $B$-module. In fact, this model also gives an alternative derivation
of the Frobenius character formula of $H_n(q)$ \cite{Ram}.

Note that the general linear group $G$ acts on $\mathrm{Ind}^G_B 1$ by left multiplication, which commutes with the natural action of
the Iwahori-Hecke algebra $H_n(q)$. By Green's theory the multiplicity of the irreducible $H_q(n)$-module appearing in $\mathrm{Ind}^G_B 1$ indexed by $\lambda$ is
controlled by the Kostka-Foulkes
polynomial $K_{\la\mu}(t)$, and the latter is exactly the irreducible character value $\chi^{\lambda}(u_{\mu})$ of $\mathrm{GL}_n(\mathbb F_q)$ at the unipotent element $u_{\mu}$ in type $\mu$.
Halverson and Ram \cite{HR} derived a combinatorial formula for the $(G, H_q(n))$-bitrace of the permutation module $\mathrm{Ind}^G_B 1$
using the Bruhat decomposition. In this paper, we will
derive an iterative formula for the bitrace on $\mathrm{Ind}^G_B 1$ and then a general formula for the bitrace.
Based on the iterative formula we also give a table of the bitrace for $n\leq 5$.

The paper is naturally divided into three parts. In Sect. \ref{S:HL} we first recall the vertex operator realization of the
Hall-Littlewood functions and express the Green polynomial $X^{\la}_{\mu}(t)$ as the transition coefficients between
the Hall basis and the power-sum basis.
Using the technique of vertex operators, we derive a useful iterative formula for $X^{\la}_{\mu}(t)$. Then we derive a general formula of
$X^{\la}_{\mu}(t)$ as well as several compact formulas in special cases. In Sect. \ref{S:MN} we derive a Murnaghan-Nakayama type formula for
the Green polynomial by using a straightening formula of the Hall-Littlewood functions indexed by compositions.
Finally in Sect. \ref{S:Bitr} we compute the bitrace of the finite general linear group $G$ and the Iwahori-Hecke algebra of type $A$ on the permutation module of $G$, and derive an iterative formula as well as the general formula. The paper is concluded with a table of the bitrace for $n\leq 5$.

\section{Vertex operator realization of Hall-Littlewood polynomials}\label{S:HL}

A composition $\lambda=(\lambda_1,\lambda_2,\ldots,\lambda_l)$, denoted by $\lambda\models n$,
is a sequence of nonnegative integers $\la_i$ (the parts) that sum up to $n$. If the sequence is
weakly decreasing, then $\la$ is called a partition and denoted as $\la\vdash n$.
The total sum $\sum_i\la_i=n$ is the weight of $\lambda$ and the number of (nonzero) parts is denoted by $l(\lambda)$. A partition $\lambda$ of weight $n$ is usually denoted by
$\lambda \vdash n$, and the
set of partitions will be denoted by $\mathcal P$. Sometimes
$\lambda$ is arranged in the ascending order: $\lambda=(1^{m_1}2^{m_2}\cdots)$ with $m_i$ being the multiplicity of $i$ in $\lambda$.
For partition $\lambda$, let $z_{\lambda}=\prod_{i\geq 1}i^{m_i(\lambda)}m_i(\lambda)!$ and denote
\begin{align}\label{e:zt}
z_{\lambda}(t)=\frac{z_{\lambda}}{\prod_{i\geq 1}(1-t^{\lambda_i})}
%n(\lambda)=\sum_{i=1}^{l(\lambda)}(i-1)\la_i.
\end{align}

The Young or Ferrers diagram of partition $\la$ is the diagram of $l(\lambda)$ rows of boxes aligned to the left where the $i$th row consists
of $\la_i$ boxes. The partition $\la'=(\la_1', \ldots, \la_{\la_1}')$ corresponding to the reflection of
the Young diagram of $\la$ along the diagonal is called the dual partition of $\lambda$.

The juxtaposition $\la\cup\mu$ of partitions $\la$ and $\mu$ is defined as the union of all parts of $\la$ and $\mu$ and then
arranged in the descending order.

Let $\Lambda$ be the ring of symmetric functions over the ring of integers. Let $F=\mathbb Q(t)$ be the field of rational functions in $t$,
and we will be mainly working with the ring $V=\Lambda_F$. The space $\Lambda$ has several well-known bases indexed by partitions: elementary symmetric functions, monomial symmetric functions, homogeneous symmetric functions, and Schur functions. The set of power sum symmetric functions is a linear
basis of $\Lambda_{\mathbb Q}$. Here
the $n$th degree power-sum symmetric function $p_n=\sum_ix_i^n$, and the power sum function $p_{\lambda}=p_{\lambda_1}p_{\lambda_2}\cdots$.
Using the degree gradation, $V$ becomes a graded ring
\begin{align}
V=\bigoplus_{n=0}^{\infty} V_n.
\end{align}
A linear operator $A$ is of degree $n$ if $A(V_m)\subset V_{m+n}$.

The space $V$ is equipped with the Hall-Littlewood bilinear form $\langle\ , \ \rangle$ defined by
\begin{align}\label{e:form}
\langle p_{\lambda}, p_{\mu}\rangle=\delta_{\lambda\mu}z_{\lambda}(t).
\end{align}
As $\{z_{\lambda}(t)^{-1}p_{\lambda}\}$ is the dual basis of the power sum basis, the dual operator of the multiplication operator $p_n$
is
the differential operator $p_n^* =\frac{n}{(1-t^n)}\frac{\partial}{\partial p_n}$ of degree $-n$. Note that * is $\mathbb Q(t)$-linear and  anti-involutive satisfying
\begin{equation}
\langle H_nu, v\rangle=\langle u, H_n^*v\rangle
\end{equation}
for $u, v\in V$.

We now recall the vertex operator realization of the Hall-Littlewood symmetric functions from \cite{Jing1}.

The \textit{vertex operators}
$H(z)$ and its dual $H^*(z)$ are $t$-parameterized linear maps: $V\longrightarrow V[[z, z^{-1}]]$ defined by
\begin{align}
\label{e:hallop}
H(z)&=\mbox{exp} \left( \sum\limits_{n\geq 1} \dfrac{1-t^{n}}{n}p_nz^{n} \right) \mbox{exp} \left( -\sum \limits_{n\geq 1} \frac{\partial}{\partial p_n}z^{-n} \right)\\ \notag
&=\sum_{n\in\mathbb Z}H_nz^{n},\\
H^*(z)&=\mbox{exp} \left(-\sum\limits_{n\geq 1} \dfrac{1-t^{n}}{n}p_nz^{n} \right) \mbox{exp} \left(\sum \limits_{n\geq 1} \frac{\partial}{\partial p_n}z^{-n} \right)\\ \notag
&=\sum_{n\in\mathbb Z}H^*_nz^{-n}.
\end{align}
Here $V[[z, z^{-1}]]=F[[z, z^{-1}]]\otimes V$ is the vector space over formal Laurent series in $z$. The components $H_n$ and $H_{-n}^*$ are endomorphisms of $V$ with degree $n$, thus
$H_{-n}$ and $H_{n}^*$ are annihilation operators for $n>0$. We collect their relations as follows.

\begin{prop}\cite{Jing1} The operators $H_n$ and $H_n^*$ satisfy the following relations
\begin{align}\label{e:com1}
H_{m}H_n-tH_nH_m&=tH_{m+1}H_{n-1}-H_{n-1}H_{m+1},\\ \label{e:com2}
H^*_{m}H^*_n-tH^*_nH^*_m&=tH^*_{m-1}H^*_{n+1}-H^*_{n+1}H^*_{m-1},\\ \label{e:com3}
H_{m}H^*_n-tH^*_nH_m&=tH_{m-1}H^*_{n-1}-H^*_{n-1}H_{m-1}+(1-t)^2\delta_{m, n},\\ \label{e:com4}
H_{-n}. 1&=\delta_{n, 0}, \qquad H_{n}^{*}. 1=\delta_{n, 0},
\end{align}
where $\delta_{m, n}$ is the Kronecker delta function.
\end{prop}

We remark that the indexing of $H_m$ and $H^*_n$ is different from that of \cite{Jing1}, where $H_n$ was denoted as $H_{-n}$ for instance.

As the vacuum vector $1$ is annihilated by $p_n^*$, we have that
\begin{equation}\label{e:qfcn1}
H(z).1=\exp(\sum_{n=1}^{\infty}\frac{1-t^n}np_nz^n)=\sum_{n=0}^{\infty}q_nz^n
\end{equation}
where $q_n$ is a symmetric function of degree $n$ in $V$, called the Hall-Littlewood polynomial associated with one-row
partition $(n)$:
\begin{equation}\label{e:qfcn2}
q_n=H_n.1=\sum_{\lambda\vdash n}\frac1{z_{\lambda}(\lambda)}p_{\lambda}.
\end{equation}

The proposition implies that %In particular, one has that
\begin{align} \label{e:com5}
H_{n}H_{n+1}&=tH_{n+1}H_{n},\\ \label{e:com6}
H_{n}^{*}H_{n-1}^{*}&=tH_{n-1}^{*}H_{n}^{*},\\ \label{e:com7a}
\langle H_n.1, H_n.1\rangle&=\sum_{\lambda\vdash n}\frac{1}{z_{\la}(t)}=1-t, \qquad n>0\\ \label{e:com7b}
\langle H_n.1, H^*_{-n}.1\rangle&=\sum_{\lambda\vdash n}\frac{(-1)^{l(\lambda)}}{z_{\la}(t)}=t^n-t^{n-1}, \qquad n>0
\end{align}
where the last two identities follow from \eqref{e:com3} and \eqref{e:com1} by induction.

Note that $H_n.1=q_n(t)$ can be generalized to all situations as the vertex operator realization of the Hall-Littlewood functions \cite{M}. For
each partition $\lambda$, denote $q_{\lambda}=q_{\lambda_1}q_{\lambda_2}\cdots$, then
$\{q_{\lambda}\}$ also forms a basis of $V$.
\begin{thm} \cite{Jing1} \label{t:HL} Let $\lambda=(\lambda_{1},\ldots ,\lambda_{l})$ be a partition.
The vertex operator products  $H_{\lambda_{1}}\cdots H_{\lambda_{l}}. 1$ is the
Hall-Littlewood function $Q_{\la}(t)$:
\begin{equation}\label{e:HL}
H_{\lambda_{1}}\cdots H_{\lambda_{l}}. 1=Q_{\la}(t)=
\prod\limits_{i<j} \dfrac{1-R_{ij}}{1-tR_{ij}}q_{\lambda_{1}}\cdots q_{\lambda_{l}}
\end{equation}
where the raising operator $R_{ij}q_{\la}=q_{(\la_{1},\ldots ,\la_{i}+1,\ldots ,\la_{j}-1,\ldots , \la_{l})}$. Moreover,
$H_{\lambda}.1=H_{\lambda_1}\cdots H_{\lambda_l}.1$ are orthogonal in $V$:
\begin{align}\label{e:orth}
\langle H_{\lambda}.1, H_{\mu}.1\rangle=\delta_{\lambda\mu}b_{\lambda}(t),
\end{align}
where $b_{\lambda}(t)=(1-t)^{l(\lambda)}\prod_{i\geqslant 1}[m_i(\lambda)]!$ and $[n]=\frac{1-t^n}{1-t}$.
\end{thm}

%The vector $H_{\lambda}.1$ is exactly the Hall-Littlewood polynomial Q_{\lambda}(t)$ associated with $\lambda$ \cite{M}.
As a result, the transition matrix
between the bases $\{p_{\lambda}\}$ and $\{H_{\lambda}\}$ gives rise to Green's polynomials. More precisely,
for $\lambda,\mu\vdash n$, let $X^{\lambda}_{\mu}(t)$ be the coefficient of
$P_{\lambda}=b_{\lambda}(t)^{-1}Q_{\lambda}(t)$ in $p_{\mu}$:
\begin{equation}\label{e:greenp}
p_{\mu}=\sum\limits_{\lambda} X^{\lambda}_{\mu}(t)P_{\lambda}(t).
\end{equation}
It is known that $X^{\la}_{\mu}(t)$ is a polynomial in $t$ of degree $n(\lambda)$, and the \textit{Green polynomials} are defined as $Q^{\la}_{\mu}(t)=t^{n(\la)}X^{\lambda}_{\mu}(t^{-1})$
for all partitions $\lambda, \mu$ of the
same weight \cite{G, M}. In the following we simply regard  $X^{\lambda}_{\mu}(t)$ as Green's polynomials.

Using theorem \ref{t:HL}, we can write Green's polynomials as:
\begin{equation}\label{e:greenp2}
X^{\lambda}_{\mu}(t)= \langle H_{\la}.1, p_{\mu} \rangle.
\end{equation}
Thus $X^{\la}_{\mu}(t)=0$ unless $|\la|=|\mu|$, and $X^{(n)}_{\la}(t)%=\langle p_{\lambda}, H_n.1\rangle
=\delta_{n, |\lambda|}$ by \eqref{e:qfcn2}.

%Note that we already used the notation $S_{\la}=S_{\la}.1$ for simplicity.
Now let's discuss how to compute $X_{\lambda}^{\mu}(t)$. Consider the following linear
maps $V\longrightarrow V[[z, z^{-1}]]$:
\begin{align}\label{e:p1}
P(z)&=\sum_{n\geq 1}p_nz^{n},\\ \label{e:p2}
P^*(z)&=\sum_{n\geq 1}p^*_nz^{-n},
\end{align}
where the operator $p_n$ and the dual $p_{n}^*$ are of degree $n$ and $-n$ respectively.

The normal ordering of vertex operators are defined as usual, so %We have the notion of normal ordering product defined as usual, i.e.,
\begin{align*}
:H^*(z)P(w):&=P(w)H^*(z),\\
:H(z)P^{*}(w):&=P^*(w)H(z).
\end{align*}

By the usual techniques of vertex operators, we have the following operator product expansions:
\begin{align}
\label{e:hallop1}
H^*(z)P(w)&=P(w)H^*(z)+H^*(z)\frac{w}{z-w}, %\sum\limits_{n\geq1}\left(\frac{w}{z}\right)^{n}
\\ \label{e:hallop2}
P^{*}(z)H(w)&=H(w)P^{*}(z)+H(w)\frac{w}{z-w}.%\sum\limits_{n\geq1}\left(\frac{z}{w}\right)^{n}
\end{align}

Taking coefficients of the above expressions, we immediately get the following commutation relations.
\begin{prop}\label{p:rel1}%{Up Down}
The commutation relations between the Hall-Littlewood vertex operators and power sum operators are:
\begin{align}\label{e:rel1}
H^{*}_{m}p_{n}=p_{n}H^{*}_{m}+H^{*}_{m-n},\\ \label{e:rel2}
p^{*}_{m}H_{n}=H_{n}p^{*}_{m}+H_{n-m}.
\end{align}
\end{prop}

To proceed we need some notations. For each partition $\la=(\la_1, \ldots, \la_l)$, we define that
\begin{align}
\la^{[i]}=(\la_{i+1}, \cdots, \la_l), \qquad i=0, 1, \ldots, l
\end{align}
So $\la^{[0]}=\la$ and $\la^{[l]}=\emptyset$. We define a subpartition $\tau$ of $\la$,
denoted $\tau \lhd \lambda$, if the parts of $\tau$ are
also parts of $\lambda$, i.e. $\tau=(\lambda_{i_{1}}, \ldots, \lambda_{i_{s}})$ for some $1\leq i_1<\cdots<i_s\leq l$.
Note that $\tau$ could be $\emptyset$ or $\lambda$.

Let $D^{(i)}(\la)$ be the number of subpartitions of $\la$ with weight $i$, then the generating function
of subparitions of $\lambda$ is given by
\begin{align}\notag
D_t(\la)&=\sum_{i\geq 0}D^{(i)}(\lambda)t^i=\sum_{\tau\lhd\la}t^{|\tau|}\\\label{e:partial}
&=(1+t)^{m_1(\la)}(1+t^2)^{m_2(\la)}\cdots=\prod_{i\geq 1}[2]_{t^i}^{m_i(\la)}.
\end{align}
In particular, the total number of subpartitions of $\la$ is $2^{l(\la)}$.

\begin{thm}\label{t:iterative}
For partition $\mu\vdash n$ with $l(\mu)=l$ and integer $k$,
\begin{align}\label{e:iterative}
H_{k}^{*}p_{\mu}&=\sum\limits_{\tau \lhd \mu} p_{\tau}H^{*}_{k+\mid \tau \mid-n}=\sum\limits_{i=0}^{n-k}\sum_{\tau\lhd \mu,\tau\vdash i}p_{\tau}H^{*}_{k+i-n}\\\label{e:pH}
p^{*}_{k}H_{\mu}&=\sum\limits_{i=1}^{l}H_{\mu_{1}}\ldots H_{\mu_{i}-k}\ldots H_{\mu_{l}}.
\end{align}
\end{thm}
\begin{proof} The second relation \eqref{e:pH} follows from \eqref{e:iterative} by taking $*$.
We argue by induction on $l(\lambda)$ for the first relation. The initial step is clear. Now assume that
\eqref{e:iterative} holds for any partition with length $<l(\lambda)$, so
%\begin{align*}
%&H_{k}^{*}p_{\lambda_{2}}p_{\lambda_{3}}\cdots p_{\lambda_{l}}=\sum\limits_{\rho \lhd \lambda^{[1]}} p_{\rho}H^{*}_{k+\lambda_{1}+\mid \tau \mid-n} \\
%&H_{k-\lambda_{1}}^{*}p_{\lambda_{2}}\cdots p_{\lambda_{l}}=\sum\limits_{\rho \lhd \lambda^{[1]}} p_{\rho}H^{*}_{k+\mid \rho \mid-n}.
%\end{align*}
it follows from Proposition \ref{p:rel1} and induction hypothesis that
\begin{align*}
H_{k}^{*}p_{\lambda}&=p_{\lambda_{1}}H^{*}_{k}p_{\lambda_{2}}\cdots p_{\lambda_{l}}+H^{*}_{k-\lambda_{1}}p_{\lambda_{2}}\cdots p_{\lambda_{l}}\\
&=p_{\lambda_{1}}\sum\limits_{\rho \lhd \lambda^{[1]}} p_{\rho}H^{*}_{k+\lambda_{1}+\mid \tau \mid-n}+\sum\limits_{\rho \lhd \lambda^{[1]}} p_{\rho}H^{*}_{k+\mid \rho \mid-n}\\
&=\sum\limits_{\tau \lhd \lambda} p_{\tau}H^{*}_{k+\mid \tau \mid-n}.
\end{align*}
\end{proof}

Note that when $n>|\lambda|$ (cf. \eqref{e:iterative})
$$ H^*_np_{\lambda_{1}}p_{\lambda_{2}}\cdots p_{\lambda_{l}}=0.$$

To effectively use our result, let us also compute the following symmetric function in $V$. For $n\geq 0$, we
have $H_{n}^*.1=\delta_{n, 0}$ and
\begin{equation}\label{e:dualq}
H^*_{-n}.1=\sum_{\lambda\vdash n}\frac{(-1)^{l(\lambda)}}{z_{\lambda}(t)}p_{\lambda},
\end{equation}
which implies that $\langle p_{\lambda}, H^*_{-n}.1\rangle=(-1)^{l(\lambda)}\delta_{n, |\lambda|}$.

\begin{exmp}
\label{X(42)} Using Theorem \ref{t:iterative} and \eqref{e:com7b} we can easily compute some Green's polynomials.
\begin{align*}
X^{(42)}_{(2^{2}1^{2})}(t)
&=\langle p_{2}p_{2}p_{1}p_{1}, H_{4}H_{2}.1\rangle\\
&=\langle H^{*}_{4}p_{2}p_{2}p_{1}p_{1}, H_{2}.1\rangle \\
&=2\langle p_{2}, H_{2}.1\rangle +\langle H^{*}_{-2}.1, H_{2}.1\rangle +2\langle p_{1}H^{*}_{-1}.1, H_{2}.1\rangle +\langle p_{1}p_{1}, H_{2}.1\rangle\\
%&=2\langle 1, 1\rangle +\langle H^{*}_{-2}, H_{2}\rangle +2\langle H^{*}_{-1}, H_{1}\rangle +\langle p_{1}, H_{1}\rangle\\
&=2+t^{2}-t+2(t-1)+1\\
&=t^{2}+t+1.
\end{align*}
\end{exmp}

\begin{thm}
\label{Recurrence Formula}
For partition $\lambda,\mu\vdash n$,
\begin{align}\notag
%X^{\mu}_{\lambda}(t)=\sum\limits_{i=0}^{n-\mu_{1}}\sum_{\mbox{\tiny$\begin{array}{c}
%\tau\lhd \lambda\\
%\tau\vdash i\end{array}$}}\sum\limits_{\rho\vdash(n-\mu_{1}-i)}\frac{(-1)^{l(\rho)}}{z_{\rho}(t)}X^{\mu^{[1]}}_{\tau\cup \rho}(t).\\
X^{\la}_{\mu}(t)&=\sum_{\mbox{\tiny$\begin{array}{c}
\tau\lhd \mu\\ |\tau|\leq n-\la_1\end{array}$}}\sum_{\rho\vdash |\la^{[1]}|-|\tau|}\frac{(-1)^{l(\rho)}}{z_{\rho}(t)}X^{\la^{[1]}}_{\tau\cup\rho}(t)\\\label{e:Recurrence}
&=\sum\limits_{i=0}^{n-\la_{1}}\sum_{\mbox{\tiny$\begin{array}{c}
\tau\lhd \mu\\
\tau\vdash i\end{array}$}}\sum\limits_{\rho\vdash(n-\la_{1}-i)}\frac{(-1)^{l(\rho)}}{z_{\rho}(t)}X^{\la^{[1]}}_{\tau\cup \rho}(t).
\end{align}
%where $z_{\rho}(t)=\prod\limits_{i\geq1}\frac{i^{m_{i}(\rho)}m_{i}(\rho)!}{1-t^{\rho_{i}}}.$
\end{thm}

\begin{proof} This recurrence formula follows from  \eqref{e:iterative} and \eqref{e:dualq}.
%the following equation
%\begin{align*}
%&H_{-m}^{*}.1=\sum\limits_{\lambda\vdash m}\frac{(-1)^{l(\lambda)}}{z_{\lambda}(t)}p_{\lambda}
%\end{align*}
%where $m$ is a nature number.
\end{proof}

When $\la=(n)$, the summation is empty, so $X^{(n)}_{\mu}(t)=1$ for any $\mu\vdash n$. When $\la=(m, n)$
\begin{align*}
X^{(m, n)}_{\mu}(t)&=\sum_{\mbox{\tiny$\begin{array}{c}
\tau\lhd \mu\\ |\tau|\leq n\end{array}$}}\sum_{\rho\vdash n-|\tau|}\frac{(-1)^{l(\rho)}}{z_{\rho}(t)}\\
&=\sum_{\mbox{\tiny$\begin{array}{c}
\tau\lhd \mu\\ |\tau|< n\end{array}$}} (t^{n-|\tau|}-t^{n-|\tau|-1})+\sum_{\mbox{\tiny$\begin{array}{c}
\tau\lhd \mu\\ |\tau|=n\end{array}$}}1\\
&=(t-1)[D_{t^{-1}}(\mu)t^{n-1}]_++D^{(n)}(\mu)\\
&=(t-1)[D_t(\mu)t^{-m-1}]_++D^{(n)}(\mu),
%&=(t^m-t^{m-1})[(1+t^{-1})^{m_1}(1+t^{-2})^{m_2}(1+t^{-3})^{m_3}\cdots]_{-m}
\end{align*}
%where $[D_t(\mu)]_m$ is the coefficient of $t^m$ in $D_t(\mu)$ (see \eqref{e:partial}), and
where $[f(t)]_{+}$ is the regular
part of the function $f(t)$ in $t$. %cut-off of the polynomial $f(t)$ up to $t^m$, and $m_i$ is the multiplicity of $i$ in $\mu$.

One can use the compact formula as follows.
\begin{align*}
X^{(42)}_{(2^21^2)}(t)&=(t-1)[(1+t)^2(1+t^2)t^{-5}]_++D^{(2)}(2^21^2)\\
&=(t-1)(t+2)+3=t^2+t+1.
\end{align*}

Using the iteration and $X^{(n)}_{\mu}(t)=1$ it follows that
\begin{align*}
&X^{(\la_1, \la_2, \la_3)}_{\mu}(t)\\
&=\sum_{\mbox{\tiny$\begin{array}{c}
\tau^1\lhd \mu, |\tau^1|\leq |\la^{[1]}|\\ \rho^1\vdash |\la^{[1]}|-|\tau^1|\end{array}$}}\frac{(-1)^{l(\rho^1)}}{z_{\rho^1}(t)}\langle
p_{\rho^1\cup\tau^1}, H_{\la^{[1]}}.1\rangle\\
&=\sum_{\mbox{\tiny$\begin{array}{c}
\tau^1\lhd \mu, |\tau^1|\leq |\la^{[1]}|\\ \rho^1\vdash |\la^{[1]}|-|\tau^1|\end{array}$}}\frac{(-1)^{l(\rho^1)}}{z_{\rho^1}(t)}
(\sum_{\mbox{\tiny$\begin{array}{c}
\tau^2\lhd \rho^1\cup\tau^1\\ \rho^2\vdash |\la^{[2]}|-|\tau^2|\end{array}$}}\frac{(-1)^{l(\rho^2)}}{z_{\rho^2}(t)})\\
&=\sum_{\mbox{\tiny$\begin{array}{c}
\tau^1\lhd \mu, |\tau^1|\leq |\la^{[1]}|\\ \rho^1\vdash |\la^{[1]}|-|\tau^1|\end{array}$}}\frac{(-1)^{l(\rho^1)}}{z_{\rho^1}(t)}
([D_{t^{-1}}(\rho^1\cup\tau^1)t^{\la_3-1}]_+(t-1)+D^{(\la_3)}(\rho^1\cup\tau^1))\\
&=\sum_{\rho^1\vdash |\la^{[1]}|-|S(\mu)|}\frac{(-1)^{l(\rho^1)}}{z_{\rho^1}(t)}
([D_{t^{-1}}(\rho^1\cup S(\mu))t^{\la_3-1}]_+(t-1)+D^{(\la_3)}(\rho^1\cup S(\mu)))
\end{align*}

Let $\la$ and $\mu$ be two partitions of $n$ and $l=l(\la)$. Let $\rho^i$, $\tau^i$ be two sequences of $l-1$ partitions such that
$|\tau^i|\leq |\la^{[i]}|$ and
\begin{align*}
&\tau^1\lhd \mu, \rho^1\vdash |\la^{[1]}|-|\tau^1|;  \quad \tau^2\lhd \rho^1\cup\tau^1, \rho^2\vdash |\la^{[2]}|-|\tau^2|;
\quad \cdots\cdots; \\
&\tau^{l-1}\lhd \rho^{l-2}\cup\tau^{l-2}, \rho^{l-1}\vdash |\la^{[l-1]}|-|\tau^{l-1}|.
%&\rho^{l}\cup\tau^{l}\vdash |\la^{[l]}|, \tau^{l}\lhd \mu.
\end{align*}
One starts with a subpartition $\tau^1$ of $\mu$ with weight $\leq |\la^{[1]}|$, then picks any partition $\rho^1$ of
weight of the difference $|\la^{[1]}|-|\tau^1|$. Then one selects the next subpartition $\tau^2$ of $\tau^1\cup\rho^1$,
and picks any partition $\rho^2$ of
weight $|\la^{[2]}|-|\tau^2|$, and continue to form $\{\tau^3, \rho^3\}, \cdots$, etc. So the weights of
$\tau^i\cup\rho^i$ are decreasing as $|\la^{[i]}|$.

%So $|\tau^i|\leq |\la^{[i]}|, |\rho^i|\geq \la_{i-1}$.
By the same method, we have the general formula:
%The above iterative formula gives us the chance to compute $X_{\mu}^{\lambda}(t)$ for any partitions $\lambda, \mu$
%First, we need some notations. Let $\lambda$ be a partition and $n,k$ be two nature numbers. Denote
%\begin{align*}
%\sum^{(n,k)}_{\lambda}\triangleq \sum\limits_{i_{k}=0}^{n}\sum_{\mbox{\tiny$\begin{array}{c}
%\tau^{(k)}\lhd \lambda\\
%\tau^{(k)}\vdash i_{k}\end{array}$}}\sum\limits_{\rho^{(k)}\vdash(n-i_{k})}\
%\end{align*}
%where $\tau^{k}, \rho^{k}$ are partitions and $i_{k}\in \mathbb{N}.$

%Let $\mu,\lambda$ be partitions. Suppose $\mu=(\mu_{1},\ldots,\mu_{l}).$ Denote
%\begin{align*}
%\sum^{\mu}_{\lambda}\triangleq \sum^{(\mu_{2}+\cdots+\mu_{l},l-1)}_{\lambda}\sum^{(\mu_{3}+\cdots+\mu_{l},l-2)}_{\tau^{(l-1)}\bigcup\rho^{(l-1)}}
%\sum^{(\mu_{4}+\cdots+\mu_{l},l-3)}_{\tau^{(l-2)}\bigcup\rho^{(l-2)}}\cdots\sum^{(\mu_{l},1)}_{\tau^{(2)}\bigcup\rho^{(2)}}
%\end{align*}

\begin{thm}\label{t:Formula} Let $\la, \mu$ be two partitions. Then the Green polynomial
\begin{align}\label{e:Formula}
%X^{\la}_{\mu}(t)=\sum_{i=1}^{l-1}\sum_{\rho^i\cup\tau^i\vdash |\la^i|}\prod^{l(\mu)-1}_{j=1}\frac{(-1)^{l(\rho^{(j)})}}{z_{\rho^{(j)}}(t)}\\
X^{\la}_{\mu}(t)=\sum_{\{\rho^i\}, \{\tau^i\}}\prod^{l(\la)-1}_{j=1}\frac{(-1)^{l(\rho^{(j)})}}{z_{\rho^{(j)}}(t)}
\end{align}
where the sum runs through all sequences of $l(\la)-1$ pairs of partitions $\{\rho^i, \tau^i\}$ such that $|\tau^i|\leq |\la^{[i]}|$, $\tau^i\lhd\tau^{i-1}\cup\rho^{i-1}$ and $\rho^i\vdash |\la^{[i]}|-|\tau^i|$, where
$i=1, \cdots, l(\la)-1$ and $\tau^0\cup\rho^0=\mu$.
\end{thm}

\begin{proof} This follows from repeatedly using \eqref{e:Recurrence}, and notice that for any $\mu\vdash m$,
$X^{(m)}_{\mu}(t)=1$ by above.
%\begin{equation}
%X^{(m)}_{\mu}(t)=1
%\end{equation}
\end{proof}

\begin{lem} \label{t:Lemma}For partition $\lambda$ of $n$, we have that
\begin{align}\label{e:z1}
%\sum\limits_{\rho\vdash n}\frac{1}{z_{\rho}(t)}&=1-t\\\label{e:z2}
\sum_{\mbox{\tiny$\begin{array}{c}
\tau\lhd \lambda\\
\tau\neq \emptyset\end{array}$}}\prod\limits_{j\geq1}(t^{\tau_{j}}-1)&=t^{n}-1,\\ \label{e:z1b}
\sum_{\tau\lhd \la}\prod\limits_{j\geq1}\frac1{t^{\tau_{j}}-1}&=\frac{t^{n}}{\prod_{i\geq 1}(t^{\la_i}-1)}
\end{align}
\end{lem}

For partition $\la$, define $[\la]=\prod_{i\geq 1}(t^{\la_i}-1)$ and $[\emptyset]=1$. Then the function
$\sum_{\tau\lhd\la}[\tau]$
is strictly multiplicative for $\la$ with respect to its parts. Note that
\begin{equation*}
\sum_{\tau\lhd (i^m)}[\tau]=\sum_{j=0}^m\binom{m}{j}(t^i-1)=t^{im}.
\end{equation*}
Therefore
\begin{equation*}
\sum_{\tau\lhd \la}[\tau]=\prod_{i\geq 1}(\sum_{\tau\lhd (i^{m_i})}[\tau])=\prod_{i\geq 1}t^{im_i}=t^{|\la|}.
\end{equation*}
The other identity can be proved similarly.

We can compute more Green's polynomials, for example some well-known formulas in \cite[\S III.7]{M}.

\begin{exmp} For each partition $\lambda$ of $n$, we have that
\begin{align}\label{e:1column}
%X_{\la}^{(n)}(t)&=1\\
X_{\la}^{(1^{n})}(t)&=\frac{\prod\limits_{i=1}^{n}(t^{i}-1)}{\prod\limits_{j\geq1}(t^{\lambda_{j}}-1)}=\frac{[n]!}{[\lambda]}.
\end{align}
\end{exmp}

This can be checked by induction using Theorem \ref{t:iterative} and Lemma \ref{t:Lemma}. The initial step of $n=1$ is clear.
Now for $\la\vdash n$
\begin{align*}
X^{(1^{n})}_{\lambda}(t)&=\sum_{\mbox{\tiny$\begin{array}{c}
\tau\lhd \lambda\\
|\tau|\leq n-1\end{array}$}}\sum\limits_{\rho\vdash(n-1-|\tau|)}\frac{(-1)^{l(\rho)}}{z_{\rho}(t)}X^{(1^{n-1})}_{\tau\cup \rho}(t)\\
&=\sum_{\mbox{\tiny$\begin{array}{c}
\tau\lhd \lambda\\
|\tau|\leq n-1\end{array}$}}\sum\limits_{\rho\vdash(n-1-|\tau|)}\frac{(-1)^{l(\rho)}}
{z_{\rho}(t)}
\frac{\prod\limits_{j=1}^{n-1}(t^{j}-1)}{\prod\limits_{j\geq1}(t^{\tau_{j}}-1)(t^{\rho_{j}}-1)}\\
&=\sum_{\mbox{\tiny$\begin{array}{c}
\tau\lhd \lambda\\|\tau|\leq n-1\end{array}$}}\sum\limits_{\rho\vdash(n-1-|\tau|)}\frac1{z_{\rho}}\frac{\prod\limits_{j=1}^{n-1}(t^{j}-1)}
{\prod\limits_{j\geq1}(t^{\tau_{j}}-1)}\\
&=\sum_{\mbox{\tiny$\begin{array}{c}
\tau\lhd \lambda\\
|\tau|\leq n-1\end{array}$}}\frac{\prod\limits_{j=1}^{n-1}(t^{j}-1)}{\prod\limits_{j\geq1}(t^{\tau_{j}}-1)}
=\frac{\prod\limits_{i=1}^{n}(t^{i}-1)}{\prod\limits_{j\geq1}(t^{\lambda_{j}}-1)},
\end{align*}
where the last identity has used \eqref{e:z1b}.

Summarizing the above, we have that
\begin{thm}\label{t:two;hook;three}
For partition $\lambda\vdash n$, one have
\begin{align}\label{e:two}
&X^{(n-k, k)}_{\lambda}(t)=\sum_{\mbox{\tiny$\begin{array}{c}
\tau\lhd \lambda, |\tau|\leq k\\\rho\vdash(k-|\tau|)\end{array}$}}\frac{(-1)^{l(\rho)}}{z_{\rho}(t)}\\\label{e:hook}
&X^{(k, 1^{n-k})}_{\mu}(t)=\frac{\prod\limits_{i=1}^{n-k}(t^{i}-1)}{\prod\limits_{j\geq1}(t^{\mu_{j}}-1)}
\sum_{\mbox{\tiny$\begin{array}{c}
\tau\lhd \mu\\
|\tau|\geq k\end{array}$}}\prod\limits_{j\geq 1}(t^{\tau_{j}}-1)\\ \label{e:three}
&X^{(k_1,k_{2},k_{3})}_{\lambda}(t)= %\notag\\
\sum_{\mbox{\tiny$\begin{array}{c}
\tau\lhd \lambda,|\tau|\leq k_2+k_3\\\rho\vdash(k_{2}+k_{3}-|\tau|)\end{array}$}}\sum_{\mbox{\tiny$\begin{array}{c}
\nu\lhd (\tau\cup\rho), |\nu|\leq k_3\\\xi\vdash(k_{3}-|\nu|)\end{array}$}}\frac{(-1)^{l(\xi)+l(\rho)}}{z_{\xi}(t)z_{\rho}(t)}\\ \notag
&X^{(h_{1},h_{2},1^{n-h_{1}-h_{2}})}_{\lambda}(t)=\\ \label{e:general hook}
&\sum_{\mbox{\tiny$\begin{array}{c}
\tau\lhd \lambda, |\tau|\leq n-h_1\\\rho\vdash(n-h_{1}-|\tau|)\end{array}$}}\sum_{\mbox{\tiny$\begin{array}{c}
\mu\lhd (\tau\cup\rho)\\h_2\leq|\mu|\leq n-h_1\end{array}$}}\frac{\prod\limits_{i=1}^{n-h_{1}-h_{2}}(t^{i}-1)
\prod\limits_{l\geq1}(t^{\mu_{l}}-1)}{\prod\limits_{j\geq1}(t^{\tau_{j}}-1)z_{\rho}}.
\end{align}
\end{thm}
\begin{proof} The first identity follows from \eqref{e:Recurrence}. % and \eqref{e:1row}.
The second identity follows from \eqref{e:Recurrence} \eqref{e:z1} and \eqref{e:1column}. The third identity follows from \eqref{e:Recurrence} and \eqref{e:two}.
The last identity follows from \eqref{e:Recurrence} and \eqref{e:hook}.
\end{proof}

\begin{rem} Morita \cite{Mo} has given a different formula for the hook case at the root of unity.
\end{rem}

\begin{exmp}
Given $\lambda=(2^{2},1^{2})$ and $\mu=(3,1^{3})$, our formula says that
\begin{align*}
&X^{(3,1^{3})}_{(2^{2},1^{2})}(t)\\
&=\frac{(t-1)(t^{2}-1)(t^{3}-1)}{(t^{2}-1)^{2}(t-1)^{2}}[4(t^{2}-1)(t-1)+(t^{2}-1)(t^{2}-1)\\
&+2(t^{2}-1)(t-1)(t-1)+2(t^{2}-1)(t^{2}-1)(t-1)]\\
&=(t^{3}+t^{2}+2t+2)(t^{3}-1).
\end{align*}
\end{exmp}

\section{A Murnaghan-Nakayama rule}\label{S:MN}
Let $k$ be a nature number and $\mu$ be a partition, we consider $p^{*}_{k}H_{\mu}$.
To do this, we need to express $H_{\lambda}$, $\lambda\vDash n$ in terms of the basis elements $H_{\mu}$, $\mu\in\mathcal P$.

For any $m, n$, repeatedly using \eqref{e:com1} gives that
\begin{align*}
[H_m, H_n]_t&=(t^2-1)H_{n-1}H_{m+1}-t[H_{n-2}, H_{m+2}]_t\\
&=\sum_{i=1}^{s-1}t^{i-1}(t^2-1)H_{n-i}H_{m+i}-t^{s-1}[H_{n-s}, H_{m+s}]_t.
\end{align*}
%In particular,
For $m<n$, let $\epsilon =0, 1$ be the parity of $n-m$, i.e. $\epsilon \equiv n-m\, (mod\, 2)$, then
$[\frac{n-m}2]=\frac{n-m-\epsilon}2$ and % then we have
\begin{align*}
H_{m}H_{n}&=tH_{n}H_{m}+\sum\limits_{i=1}^{[\frac{n-m}2]-1}(t^{i+1}-t^{i-1})H_{n-i}H_{m+i}\\
%-t^{[\frac{n-m}2]}[H_{\frac{n+m+\epsilon}2},H_{\frac{n+m-\epsilon}2}]_t\\
%&tH_{n}H_{m}+\sum\limits_{i=1}^{[\frac{n-m}2]-1}(t^2-1)t^{i-1}H_{n-i}H_{m+i}
&+t^{[\frac{n-m}2]-1}(t^{1+\epsilon}-1)H_{\frac{n+m+\epsilon}2}H_{\frac{n+m-\epsilon}2}
\end{align*}
i.e. for $n-i>m+i$, the coefficient of $H_{n-i}H_{m+i}$ is $(t^{i+1}-t^{i-1})$; if $n-i=m+i$
the coefficient of $H_{n-i}H_{m+i}$ is $(t^{i}-t^{i-1})$.
%In general, if there are
%$k$ transpositions changing $H_{\lambda}$ to $H_{\la+i_1R^{i_1}+\cdots+ i_kR^{i_k}}$, where $R_i=R_{i, i+1}$ is the
% raising operator on the composition.
% the attached factor is $t^{i_1+\cdots+i_k-k}(t^2-1)^k$.
Note that at any stage if $\la_i>\la_{i+1}+\cdots$, $H_{(\cdots, -\la_i, \la_{i+1}, \la_{i+2}, \cdots)}=0$.

For a composition $\lambda$ such that
$\lambda_i<\lambda_{i+1}$, let $S_{i,a}$ be the transformation $(\la_1, \cdots, \la_i, \la_{i+1}, \cdots)\mapsto (\la_1, \cdots, \la_{i+1}-a, \la_{i}+a, \cdots)$, where
$0\leq a\leq [\frac{\la_{i+1}-\la_i}2]$.
Define
\begin{equation}\label{e:straight}
C(S_{i,a})=\begin{cases} t & a=0\\ t^{a+1}-t^{a-1} & 1\leq a< [\frac{\la_{i+1}-\la_i}2] \\ t^{a+\epsilon}-t^{a-1} & a= [\frac{\la_{i+1}-\la_i}2]
\end{cases}
\end{equation}
where $\epsilon$ is the parity of $\la_i-\la_{i+1}$.
For $\underline{i}=(i_1, \ldots, i_r)$ and $\underline{a}=(a_1, \ldots, a_r)$ define
\begin{equation}
C(S_{\underline{i}, \underline{a}})=C(S_{i_1, a_1})C(S_{i_2, a_2})\dots C(S_{i_r, a_r})
\end{equation}
where the product order follows the action order of $S_{i_1, a_1}S_{i_2, a_2}\cdots S_{i_r, a_r}\la$ from right to left.
The following are two special cases:
(i) if $t=0$, then $C(\underline{i}, \underline{a})=0$ unless all $a_i=1$. When all $a_i=1$ for $1\leq i\leq r$ (which only happens when $\la_{i+1}-\la_i\geq 2$), then $C(\underline{i}, \underline{a})=(-1)^r$;
%which is possible only when $\la_{i+1}-\la_i\geq 2$.
(ii) if $t=-1$, then $C(\underline{i}, \underline{a})=0$ unless all $a_i=0$ in which $C(\underline{i}, \underline{a})=(-1)^r$.

\begin{prop}\label{p:straight} Suppose $\la$ is a composition, then
\begin{equation}\label{e:straight2}
H_{\la}=\sum_{\underline{i}, \underline{a}} C(S_{\underline{i}, \underline{a}}) H_{S_{i_1,a_1}S_{i_2, a_2}\cdots S_{i_r, a_r}\la}
\end{equation}
summed over $\underline{i}=(i_1, \ldots, i_r), \underline{a}=(a_1, \cdots, a_r)\in \mathbb Z_+^r$ such that
$S_{i_1,a_1}S_{i_2, a_2}$ $\cdots S_{i_r, a_r}\la\in\mathcal P$.
\end{prop}

The following is a Murnaghan-Nakayama rule for the Green polynomial. The result generalizes a formula of Morris \cite{Mor2} which corresponds to our result in the case of $l(\lambda)=2$.
\begin{thm} Let $\la, \mu\in\mathcal P_n$, then
\begin{equation}
X^{\la}_{\mu}(t)=\sum_{j=1}^{l(\mu)}\sum_{\underline{i}, \underline{a}} C(S_{\underline{i}, \underline{a}}) X^{S_{i_1,a_1}S_{i_2, a_2}\cdots S_{i_r, a_r}(\la-\mu_1\varepsilon_j)}_{\mu^{[1]}}(t)
\end{equation}
summed over $\underline{i}=(i_1, \ldots, i_r), \underline{a}=(a_1, \cdots, a_r)\in \mathbb Z_+^r$ such that
$S_{i_1,a_1}S_{i_2, a_2}$ $\cdots S_{i_r, a_r}(\la-\mu_1\epsilon_j)\in\mathcal P_{n-\mu_1}$. Here $\varepsilon_j$ is the composition with
$1$ at the $j$-th position $0$ elsewhere.
\end{thm}
\begin{proof} By Prop. \ref{p:straight} it follows that
\begin{align*}
p_{\mu_1}^*H_{\la}&=\sum_{i=1}^{l(\la)}H_{\la-\mu_1\varepsilon_i}.1\\
&=\sum_{i=1}^{l(\la)}\sum_{\underline{i}, \underline{a}} C(S_{\underline{i}, \underline{a}}) H_{S_{i_1,a_1}S_{i_2, a_2}\cdots S_{i_r, a_r}( \la-\mu_1\varepsilon_i)}.1,
\end{align*}
where the sum runs through all $\underline{i}=(i_1, \ldots, i_r), \underline{a}=(a_1, \cdots, a_r)\in \mathbb Z_+^r$ such that
$S_{i_1,a_1}S_{i_2, a_2}$ $\cdots S_{i_r, a_r}(\la-\mu_1\varepsilon_j)\in\mathcal P_{n-\mu_1}$, which immediately implies the theorem.
\end{proof}

\begin{exmp}
Given $\la=(9,5,2)$ and $\mu=(8,4,2,2)$, then
%\begin{align*}
%p^{*}_{k}H_{\mu}&=t^{2}H_{5}H_{2}H_{1}+(t^{2}-t^{0})H_{4}H_{2}H_{2}+(t-t^{2})H_{3}H_{3}H_{2}+tH_{9}H_{2}H_{-3}\\
%&+(t^{2}-t^{0})H_{9}H_{1}H_{-2}+(-1)(t-t^{3})H_{9}H_{0}H_{-1}+H_{9}H_{5}H_{-6}\\
%&=t^{2}H_{5}H_{2}H_{1}+(t^{2}-1)H_{4}H_{2}H_{2}-(t^{2}-t)H_{3}H_{3}H_{2}+tH_{9}H_{2}H_{-3}\\
%&+(t^{2}-1)H_{9}H_{1}H_{-2}+(t^{3}-t)H_{9}H_{0}H_{-1}+H_{9}H_{5}H_{-6}.
%\end{align*}
\begin{align*}
p^{*}_{8}H_{\la}.1&=H_{(1,5,2)}.1+H_{(9, -3, 2)}.1+H_{(9,5,-6)}=H_{(1,5,2)}.1\\
&=C(S_2S_1)H_{S_2S_1(1, 5, 2)}.1+C(S_{1;1})H_{S_{1;1}(1,5,2)}.1+C(S_{1;2})H_{S_{1;2}(1,5,2)}.1\\
&=t^2H_{(5,2,1)}.1+(t^2-1)H_{(4,2,2)}.1+(t^2-t)H_{(3,3,2)}.1.
\end{align*}
Therefore
\begin{align*}
X^{(9,5,2)}_{(8,4,2,2)}(t)=t^2X^{(5,2,1)}_{(4,2,2)}(t)+(t^2-1)X^{(4,2,2)}_{(4,2,2)}(t)+(t^2-t)X^{(3,3,2)}_{(4,2,2)}(t).
\end{align*}
\end{exmp}

\begin{exmp} Now let's consider the special case $X^{\la}_{(1^n)}(t)$. It is easy to see that for
$\la=(1^{m_1}2^{m_2}\cdots)$
\begin{equation*}
p^*_1H_{\la}=\sum_{i=1}^{\la_1'}H_{\la-\varepsilon_i}.1=\sum_{i\geq 1}[m_i]H_{(1^{m_1}\cdots (i-1)^{m_{i-1}+1}i^{m_i-1}\cdots )}.1
\end{equation*}
Repeating the process, we have that
\begin{align*}
p^{*2}_1H_{\la}&=\sum_{i, j=1}^{\la_1'}H_{\la-\varepsilon_i-\varepsilon_j}.1\\
&=\sum_{|i-j|\geq 2}[m_i][m_j]
H_{(\cdots (i-1)^{m_{i-1}+1}i^{m_i-1}\cdots (j-1)^{m_{j-1}+1}j^{m_j-1}\cdots )}.1\\
&\quad +\sum_{|i-j|=1}[m_i][m_j]
H_{(\cdots (i-1)^{m_{i-1}+1}i^{m_i}{(i+1)}^{m_{i+1}-1}\cdots )}.1\\
&\quad +\sum_{i}[m_i][m_{i}-1]
H_{(\cdots (i-1)^{m_{i-1}+2}i^{m_i-2}\cdots )}.1
\end{align*}
Continuing in this way, $p_1^{*n}$ will eventually turn each summand into a function in $t$ and we get that
\begin{equation*}
p^{*n}_1H_{\la}.1=\sum_{T}\phi_T(t)
\end{equation*}
where $T$ runs through all standard tableaux of shape $\la$. The function $\phi_T(t)$ is defined as follows.
For a skew horizontal strip $\theta=\la-\mu$, we define
\begin{equation}
\phi_{\theta}(t)=\prod_{i\in I}[m_i(\la)]
\end{equation}
where $I$ is the set of $i$ such that $\theta'_i=1, \theta'_{i+1}=0$.
Any standard tableaux $T$ is a union of skew horizontal strips $\theta^{(i)}=\la^{(i)}-\mu^{(i)}$,
then define $\phi_T(t)=\prod_{i=1}^r\phi_{\theta^{(i)}}(t)$. As a result $X^{\la}_{(1^n)}(t)=\sum_T\phi_T(t)$,
where $T$ runs through all standard tableaux of shape $\la$.
\end{exmp}

As we mentioned before when $t=0$, $H_{\la}.1=S_{\la}.1$ is the Schur function associated with partition $\la$
and the Schur function. In this case the straightening rule \eqref{e:straight2} reduces to
$S_mS_n=-S_{n-1}S_{m+1}$.
%Our discussion leads to a simple derivation of the Murnaghan-Nakayama rule.
%We simply write $H_{\la}=S_{\la}.1$ in the case, and
This can be reformulated as follows. Let $\delta=(l-1, \cdots, 1, 0)$, we
say two $l-$tuples $\mu$ and $\lambda$ are related if $\mu+\delta=\sigma(\lambda+\delta)$ for some permutation $\sigma$. We denote by $\pi(\mu)$ the associated non-increasing integral tuple $\lambda$ of $\mu.$ If there exists an odd permutation $\sigma$ such that $\mu+\delta=\sigma(\mu+\delta),$ then we say that $\mu$ is degenerate, then
\begin{align}
S_\mu.1=
\begin{cases}
sgn(\sigma)S_{\pi(\mu)}.1& \text{if $\pi(\mu)\in\mathcal{P}$}\\
0&\text{if $\mu$ is degenerate or $\pi(\mu)\notin\mathcal{P}$}
\end{cases}
\end{align}

We can easily recover the usual Murnaghan-Nakayama rule (cf. \cite{Sa}) using vertex operators.
\begin{exmp}\label{t:M-N rule}%(The usual Murnaghan-Nakayama rule)
Let $\mu=(\mu_{1},\ldots, \mu_{l})$ be a partition and $k$ a positive integer. Then one has that
\begin{align}
p^{*}_{k}S_{\la}.1=\sum\limits_{\mu}(-1)^{ht(\la-\mu)}S_{\mu}.1
\end{align}
summed over all partitions $\mu\subset \la$ such that $\la-\mu$ is a border strip of length $k$.
\end{exmp}

\begin{proof} By \eqref{e:pH},
\begin{align*}
p^{*}_{k}S_{\la}.1=\sum\limits_{i=1}^{l}S_{\la_{1}}\cdots S_{\la_{i}-k}\cdots S_{\la_{l}}.1.
\end{align*}
Note that $S_{\la_{1}}\cdots S_{\la_{i}-k}\cdots S_{\la_{l}}.1=0$ unless $(\la_{1}+l-1,\ldots,\la_{i}-k+l-i,\ldots,\la_{l})\in \mathbb{Z}^{+}_{l}$ has no identical terms. We rearrange the sequence $\la+\delta$
in descending order, and we may assume that for some $j>i$
\begin{align*}
\la_{j}+l-j<\la_{i}-k+l-i<\la_{j-1}+l-(j-1),
\end{align*}
in which case the related partition $\mu$ is
\begin{align*}
(\la_{1},\ldots,\la_{i-1},\la_{i+1}-1,\ldots,\la_{j-2}-1,\la_{i}-k-i+j-1,\la_{j},\ldots,\la_{l})
\end{align*}
therefore $\theta=\la-\mu$ is a border strip of length $k$ and $ht(\theta)=j-i-1$, $sgn(\sigma)=(-1)^{l(\sigma)}=(-1)^{j-i-1}.$
\end{proof}

\begin{rem}
When $t=-1$, $H_{\mu}.1$ is the Schur Q-function. In this case, one can also obtain a result similar to Example \ref{t:M-N rule}.
See \cite[\S III.8, Ex.11]{M} for details.
\end{rem}

\section{Bitraces for $\mathrm{GL}_n(\mathbb{F}_q)$ and the Hecke algebra of type $A_{n-1}$}\label{S:Bitr}

 Let $H_n(q)$ be the Iwahori-Hecke algebra of the symmetric group $S_n$ and $G=\mathrm{GL}_n(\mathbb{F}_q)$ the general linear group over the finite field $\mathbb{F}_q$, where $q=p^m$. Let $W=Ind^G_B 1$ be the permutation module of $G$ induced from the Borel subgroup $B$ consisting of upper triangular matrices. Then $H_n(q)$ naturally acts on $W$ which commutes with that of $G$, so $W$ becomes a $G$-$H_n(q)$ bimodule. Following \cite{HR}, we define the bitrace of $(g,h)\in G\times H_n(q)$ on $W$ as follows:
\begin{defn}
Let $u\in G$ and $h\in H_n(q)$. The trace of the action of $uh$ on $Ind^G_B 1$ is
\begin{align}
btr(u,h)=\sum\limits_{gB\in G/B}u(gB)h|_{gB}
\end{align}
where $u(gB)h|_{gB}$ denotes the coefficient of $gB$ in $u(gB)h$.
\end{defn}

A combinatorial formula for $btr(u_{\lambda},T_{\mu})$ is known by using the explicit description of the action of $H_n(q)$ in terms of the Bruhat decomposition of $G$ \cite{HR}. We now give an algebraic iterative formula for $btr(u_{\lambda},T_{\mu}).$

%There are some interpretations of the bitrace defined above.
As a bimodule and in view of double centralizer property, $W$ decomposes itself into:
\begin{align*}
W=\bigoplus\limits_{\lambda}G^{\lambda}\otimes H^{\lambda}
\end{align*}
where $G^{\lambda}$ (resp. $H^{\lambda}$) is an irreducible $G$ (resp. $H_n(q)$)-module. Taking trace gives rise to
\begin{align}
btr(g,h)=\sum\limits_{\lambda\vdash n}\chi^{\lambda}(g)\zeta^{\lambda}(h)
\end{align}
where $\chi^{\lambda}$ (resp. $\zeta^{\lambda}$) is the irreducible character of $G$ (resp. $H_n(q))$.

By \cite{G} the irreducible character $\chi^{\lambda}$ of $G$ is given by%we have that (see \cite{G})
\begin{align}\label{e:KF}
\chi^{\lambda}(u_{\nu})=q^{n(\nu)}K_{\lambda,\nu}(q^{-1})
\end{align}
where $u_{\nu}$ is a unipotent element of $\mathrm{GL}_n(\mathbb{F}_q)$ with Jordan normal form of blocks size $\nu_i$
and $K_{\lambda,\nu}(t)$ is the Kostka-Foulkes polynomial defined by expanding the Schur function $s_{\lambda}$ %$Q_{\nu}(t)$
 in terms of the dual Hall-Littlewood functions $P_{\mu}(t)=b_{\mu}^{-1}(t)Q_{\lambda}(t)$ ($t=q^{-1}$):
\begin{align*}
%Q_{\nu}(t)=\sum\limits_{\lambda\vdash n}K_{\lambda,\nu}(t)S_{\lambda}.
s_{\lambda}=\sum\limits_{\nu\vdash n}K_{\lambda,\nu}(t)P_{\nu}(t).
\end{align*}

Also the Frobenius formula for the Hecke algebra \cite{Ram} says that %for $\tilde{\chi}^{\lambda},$ we have the
\begin{align}\label{e:Frob}
\frac{q^{|\mu|}}{(q-1)^{l(\mu)}}q_{\mu}(q^{-1})=\sum\limits_{\lambda\vdash n}\zeta^{\lambda}(T_{\gamma_{\mu}})s_{\lambda}.
\end{align}
Combining \eqref{e:KF} and \eqref{e:Frob}, we see that the bitrace is expressed as the matrix coefficient:
%the usual inner products of symmetric functions:
\begin{align}
btr(u_{\nu},T_{\gamma_{\mu}})=\frac{q^{|\mu|+n(\nu)}}{(q-1)^{l(\mu)}}\langle Q_{\nu}(q^{-1}), q_{\mu}(q^{-1}) \rangle.
\end{align}

Let $B^{\nu}_{\mu}(t)=\langle Q_{\nu}(t), q_{\mu}(t) \rangle$, %and we just consider $B^{\nu}_{\mu}(t)$.
by Theorem \ref{t:HL} it follows that
\begin{equation}\label{e:btr}
B^{\nu}_{\mu}(t)= \langle H_{\nu}.1, q_{\mu} \rangle.
\end{equation}
Thus we can use vertex operator technique to compute the bitrace as follows.

First of all, it is easy to see the operator product expansions as in \eqref{e:hallop1}-\eqref{e:hallop2}:
\begin{align}
\label{e:relations1}
H^*(z)q(w)&=q(w)H^*(z)\frac{z-tw}{z-w},
\\ \label{e:relations2}
q^{*}(z)H(w)&=H(w)q^{*}(z)\frac{z-tw}{z-w}.
\end{align}

Taking coefficients of $z^{-n}w^{m}$ in \eqref{e:relations1} and \eqref{e:relations2}, we get the following commutation relations.
\begin{prop}\label{p:rel2}
For any $m,n\in\mathbb{Z},$
\begin{align}\label{e:com3b}
H^{*}_{n}q_{m}=q_{m}H^{*}_{n}+(1-t)\sum\limits_{k=1}^{m}q_{m-k}H^*_{n-k},\\ \label{e:com4b}
q^{*}_{n}H_{m}=H_{m}q^{*}_{n}+(1-t)\sum\limits_{k=1}^{n}H_{m-k}q^*_{n-k}.
\end{align}
\end{prop}

Using the same method of Theorem \ref{t:iterative}, we immediately have the following result.
\begin{thm}\label{t:iterative2}
For partitions $\lambda,\mu\vdash n$ and integer number $k$,
\begin{align}\label{e:qH}
q_{k}^{*}H_{\nu}&=\sum\limits_{\tau\models k}(1-t)^{l(\tau)}H_{\nu-\tau}, \\ \label{e:Hq}
H^{*}_{k}q_{\mu}&=\sum\limits_{\tau\in\mathbb{Z}^{l}_{+}}(1-t)^{l(\tau)}q_{\mu-\tau}H^{*}_{k-\mid \tau \mid},
\end{align}
where $\mathbb{Z}_{+}$ is the set of non-negative integer.
\end{thm}
%\begin{proof}
%It is an analogous proof of Theorem \ref{t:iterative}.
%\end{proof}

%We are now ready to give our main result, an algebraic formula for the bitrace. First of all, we introduce some notations as follows.

Let $\mu$ be a composition and $\lambda$ be a partition. Recall \eqref{e:straight} and set
\begin{align}\label{e:def}
B(\lambda, \mu)\doteq\sum\limits_{\underline{i},\underline{a}}C(S_{\underline{i},\underline{a}})
\end{align}
summed over $\underline{i}=(i_1,i_2,\ldots,i_r)$, $\underline{a}=(a_1,a_2,\ldots,a_r)$ such that $S_{\underline{i},\underline{a}}\mu=\lambda.$

We remark that $B(\lambda, \mu)=0,$ unless $|\lambda|=|\mu|$. If $\mu_i+\mu_{i+1}+\cdots<0$ at any stage, then $B(\lambda, \mu)=0.$ And if $\mu$ is also a partition, then $B(\lambda, \mu)=\delta_{\lambda,\mu}.$ Let $\lambda, \nu$ be partitions and $\tau$ be a non-negative composition, then $B(\lambda, \nu-\tau)=0$ unless $\lambda\subset \nu.$
\begin{lem}\label{t:B}
Let $\lambda=(\lambda_1,\lambda_2,\ldots,\lambda_l), \nu=(\nu_1,\nu_2,\ldots,\nu_m)$ be partitions. If $l<m,$ $\lambda_i=\nu_i, i=1,2,\ldots,l$ and $\nu_l>\nu_{l+1},$ then we have
\begin{align*}
&\sum\limits_{\tau\models|\nu|-|\lambda|}(1-t)^{l(\tau)}B(\lambda, \nu-\tau)\\
&=\sum\limits_{\tau\models|\nu|-|\lambda|}(1-t)^{l(\tau)}B(\emptyset, (\nu_{l+1},\ldots,\nu_m)-\tau).
\end{align*}
\end{lem}
\begin{proof} This follows directly from \eqref{e:def} and the remark below it.
\end{proof}

Recall that $\mu\models n$ means $\mu$ is a composition with weight $n$. In this case, we can rewrite \eqref{e:straight2} as ($\mu\vDash n$):
\begin{align}
H_{\mu}=\sum\limits_{\lambda\vdash n}B(\lambda, \mu)H_{\lambda}.
\end{align}
For $\nu\vdash n$ and using \eqref{e:qH} it follows that
\begin{align}\label{e:qH2}
q_{k}^{*}H_{\nu}&=\sum\limits_{\lambda\in \mathcal{P}_{n-k}^{\nu}}\sum\limits_{\tau\models k}(1-t)^{l(\tau)}B(\lambda, \nu-\tau)H_{\lambda}.1
\end{align}
where $\mathcal{P}_{n}^{\nu}$ is the set $\{\lambda\vdash n\mid \lambda\subset \nu\}$. Note that $\lambda$ appears in \eqref{e:qH2} only when $\nu/\lambda$ is a horizontal $k$-strip (cf. \cite[III (5.7)]{M}), so we have proved the following result.
%We remark that by Chapter III (5.7) in \cite[III (5.7)]{M}, .
%\begin{lem}
%Let $\lambda, \nu$ be partitions and $k$ be a positive integer, $\lambda\subset \nu$, we have
%\begin{align}
%\sum\limits_{\tau\models k}(1-t)^{l(\tau)}B(\lambda, \nu-\tau)=0
%\end{align}
%unless $\nu/\lambda$ is a horizontal $k$-strips.
%\end{lem}
%\begin{proof}

%\end{proof}
%\bigskip
\begin{thm}\label{t:Biterative}
Let $\mu,\nu\vdash n,$ then the following iterative formula holds. %for $B^{\nu}_{\mu}(t)$
\begin{align}\label{e:Biterative}
B^{\nu}_{\mu}(t)=\sum\limits_{\lambda\in \mathcal{P}_{n-\mu_1}^{\nu}}\sum\limits_{\tau\models \mu_1}(1-t)^{l(\tau)}B(\lambda, \nu-\tau)B^{\lambda}_{\mu^{[1]}}(t).
\end{align}
\end{thm}
\begin{proof}
This follows from \eqref{e:qH2} and \eqref{e:btr}.
\end{proof}

We list some of the special cases of Theorem \ref{t:Biterative}. %, we can obtain the special cases for $B^{\nu}_{\mu}(t).$
\begin{exmp}
Let $\mu,\nu\vdash n,$ we have
\begin{align}\label{e:(n)}
B^{(n)}_{\mu}(t)&=(1-t)^{l(\mu)},\\\label{e:,(n)}
B^{\nu}_{(n)}(t)&=(1-t)\delta_{\nu,(n)},\\\label{e:(1^n),}
B^{(1^n)}_{\mu}(t)&=\delta_{\mu,(1^n)}\prod\limits_{i=1}^{n}(1-t^i),\\\label{e:,(1^n)}
B^{\nu}_{(1^n)}(t)&=(1-t)^nX^{\nu}_{(1^n)}(t),\\
\label{e:,two}
B^{\nu}_{(\mu_1,\mu_2)}(t)&=
\begin{cases}
(1-t)^{3-\delta_{\mu_1,\nu_2}-\delta_{\nu_2,0}}+t(1-t)^2\delta_{\mu_2,\nu_2}& \text{if $\nu\geq(\mu_1,\mu_2)$}\\
0&\text{if others.}
\end{cases}
\end{align}
\end{exmp}
\begin{proof}
\eqref{e:(n)}, \eqref{e:(1^n),}, and \eqref{e:,two}  follows from \eqref{e:Biterative} by easy induction. \eqref{e:,(n)} holds by \eqref{e:btr} and \eqref{e:orth}. \eqref{e:,(1^n)} holds by \eqref{e:btr}, \eqref{e:qfcn2} and \eqref{e:greenp2}.
\end{proof}

\begin{cor}
Let $\nu,\mu\vdash n$, $\mu>\nu,$ we have
\begin{align}
B^{\nu}_{\mu}(t)=0.
\end{align}
\end{cor}
\begin{proof}
We argue by induction on $l(\mu)$. The case $l(\mu)=1$ is \eqref{e:,(n)}. Suppose it holds for all $l(\mu)<l$
and consider $\mu=(\mu_1,\mu_2,\ldots,\mu_l), \nu<\mu$. Since $q_mq_n=q_nq_m,$ by \eqref{e:Biterative}, we have
\begin{align*}
B^{\nu}_{\mu}(t)=\sum\limits_{\lambda\in \mathcal{P}_{n-\mu_l}^{\nu}}\sum\limits_{\tau\models \mu_l}(1-t)^{l(\tau)}B(\lambda, \nu-\tau)B^{\lambda}_{(\mu_1,\ldots,\mu_{l-1})}(t)
\end{align*}
By induction hypothesis and the remark above Lemma \ref{t:B}, we have
\begin{align*}
B^{\nu}_{\mu}(t)=\prod\limits_{i=1}^{l-1}\delta_{\nu_i,\mu_i}\sum\limits_{\tau\models \mu_l}(1-t)^{l(\tau)}B((\mu_1,\ldots,\mu_{l-1}), \nu-\tau)B^{(\mu_1,\ldots,\mu_{l-1})}_{(\mu_1,\ldots,\mu_{l-1})}(t)
\end{align*}
It suffices to consider the case $\nu_i=\mu_i, i=1,2,\ldots,l-1.$ Then we have $\nu_l<\mu_l, \nu_{l+1}>0,$ so
for $\bar{\mu}=(\mu_1, \ldots, \mu_{l-1})$
\begin{align*}
B^{\nu}_{\mu}(t)&=\sum\limits_{\tau\models \mu_l}(1-t)^{l(\tau)}B(\bar{\mu}, (\bar{\mu},\nu_l,\cdots)-\tau)B^{\bar{\mu}}_{\bar{\mu}}(t)\\
&=\sum\limits_{\tau\models \mu_l}(1-t)^{l(\tau)}B(\emptyset, (\nu_l,\nu_{l+1},\cdots)-\tau)B^{\bar{\mu}}_{\bar{\mu}}(t) \quad(\text{by Lemma \ref{t:B}})\\
&=B^{(\nu_l,\nu_{l+1},\cdots)}_{(\mu_l)}B^{\bar{\mu}}_{\bar{\mu}}(t)\\
&=0.
\end{align*}
\end{proof}

{\bf Tables for $btr(u_{\nu},T_{\gamma_{\mu}}), n\leq 5$}. Here $[n]=1+\cdots+q^{n-1}$

\begin{table}[H]

\caption{\label{tab:2}n=2}

 \begin{tabular}{|c|c|c|c|}

  \hline

 \tabincell{c}{$\mu\backslash \nu$ } & $(2)$ & $(1^2)$ \\

  \hline

$(2)$ & \tabincell{c}{$q$} & \tabincell{c}{$0$} \\

\hline

$(1^2)$ & \tabincell{c}{$1$} & \tabincell{c}{$[2]$} \\

\hline

 \end{tabular}

\end{table}

\begin{table}[H]

 \centering

\caption{\label{tab:3}n=3}

 \begin{tabular}{|c|c|c|c|c|}

  \hline

 \tabincell{c}{$\mu\backslash \nu$ } & $(3)$ & $(2,1)$ & $(1^3)$ \\

  \hline

$(3)$ & \tabincell{c}{$q^2$} & \tabincell{c}{$0$} & \tabincell{c}{$0$}  \\

\hline

$(2,1)$ & \tabincell{c}{$q$} & \tabincell{c}{$q^2$} & \tabincell{c}{$0$} \\

\hline
$(1^3)$   & \tabincell{c}{$1$}    & \tabincell{c}{$2q+1$}   & \tabincell{c}{$\prod\limits_{i=1}^3[i]$}   \\

  \hline

 \end{tabular}

\end{table}

\begin{table}[H]

 \centering

\caption{\label{tab:4}n=4}

 \begin{tabular}{|c|c|c|c|c|c|c|}

  \hline

 \tabincell{c}{$\mu\backslash \nu$ } & $(4)$ & $(3,1)$ & $(2^2)$ & $(2,1^2)$ & $(1^4)$ \\

  \hline

$(4)$ & \tabincell{c}{$q^3$} & \tabincell{c}{$0$} & \tabincell{c}{$0$} & \tabincell{c}{$0$} & \tabincell{c}{$0$}  \\

\hline

$(3,1)$ & \tabincell{c}{$q^2$} & \tabincell{c}{$q^3$} & \tabincell{c}{$0$} & \tabincell{c}{$0$} & \tabincell{c}{$0$}  \\

\hline

$(2^2)$ & \tabincell{c}{$q^2$}  & \tabincell{c}{$q^3-q^2$}  & \tabincell{c}{$q^4+q^3$}  & \tabincell{c}{$0$}  & \tabincell{c}{$0$} \\

  \hline
$(2,1^2)$  & \tabincell{c}{$q$}    & \tabincell{c}{$2q^2$}    & \tabincell{c}{$q^3+q^2$}    & \tabincell{c}{$q^4+q^3$}  & \tabincell{c}{$0$}  \\

  \hline
$(1^4)$   & \tabincell{c}{$1$}    & \tabincell{c}{$3q+1$}    & \tabincell{c}{$(2q+1)[2]$}    & \tabincell{c}{$(3q^2+2q+1)[2]$}  & \tabincell{c}{$\prod\limits_{i=1}^4[i]$}   \\

  \hline

 \end{tabular}

\end{table}

\begin{table}[H]

 \centering

\caption{\label{tab:5}n=5}

 \begin{tabular}{|c|c|c|c|c|c|c|c|c|}

  \hline

 \tabincell{c}{$\mu\backslash \nu$ } & $(5)$ & $(4,1)$ & $(3,2)$ & $(3,1^2)$ & $(2^2,1)$ & $(2,1^3)$ & $(1^5)$ \\

  \hline

$(5)$ & \tabincell{c}{$q^4$} & \tabincell{c}{$0$} & \tabincell{c}{$0$} & \tabincell{c}{$0$} & \tabincell{c}{$0$} & \tabincell{c}{$0$} & \tabincell{c}{$0$} \\

\hline

$(4,1)$ & \tabincell{c}{$q^3$} & \tabincell{c}{$q^4$} & \tabincell{c}{$0$} & \tabincell{c}{$0$} & \tabincell{c}{$0$} & \tabincell{c}{$0$} & \tabincell{c}{$0$} \\

\hline

$(3,2)$ & \tabincell{c}{$q^3$}  & \tabincell{c}{$q^4-q^3$}  & \tabincell{c}{$q^5$}  & \tabincell{c}{$0$}  & \tabincell{c}{$0$}  & \tabincell{c}{$0$}  & \tabincell{c}{$0$} \\

\hline

$(3,1^2)$ & \tabincell{c}{$q^2$}  & \tabincell{c}{$2q^3$}  & \tabincell{c}{$q^4$}  & \tabincell{c}{$q^5+q^4$}  & \tabincell{c}{$0$}  & \tabincell{c}{$0$}  & \tabincell{c}{$0$} \\

  \hline
$(2^2,1)$  & \tabincell{c}{$q^2$}    & \tabincell{c}{$2q^3-q^2$}    & \tabincell{c}{$2q^4$}    & \tabincell{c}{$q^5-q^3$}    & \tabincell{c}{$q^5[2]$}    & \tabincell{c}{$0$}    & \tabincell{c}{$0$}  \\

  \hline
$(2,1^3)$  & \tabincell{c}{$q$}    & \tabincell{c}{$3q^2$}    & \tabincell{c}{$3q^3+q^2$}    & \tabincell{c}{$3q^3[2]$}    & \tabincell{c}{$q^3(2q+1)[2]$}    & \tabincell{c}{$q^4[3][2]$}    & \tabincell{c}{$0$}  \\

  \hline
$(1^5)$   & \tabincell{c}{$1$}    & \tabincell{c}{$4q+1$}    & \tabincell{c}{$5q^2+4q+1$}    & \tabincell{c}{$(6q^2+3q$\\ $+1)[2]$}    & \tabincell{c}{$(5q^3+6q^2$\\ $+3q+1)[2]$}    & \tabincell{c}{$(4q^3+3q^2+$\\ $2q+1)[3][2]$}    & \tabincell{c}{$\prod\limits_{i=1}^5[i]$}   \\

  \hline

 \end{tabular}

\end{table}

\vskip30pt \centerline{\bf Acknowledgments}
The project is partially supported by
Simons Foundation grant No. 523868 and NSFC grant No. 11531004.
\bigskip

\bibliographystyle{plain}

\end{document}